\documentclass[11pt]{amsart}
\usepackage[utf8]{inputenc}
\usepackage[english]{babel}
\usepackage{amsmath, amsfonts, amssymb, amsthm, stmaryrd}
\usepackage{comment}

\usepackage[a4paper,top=3cm,bottom=2cm,left=3.5cm,right=3.5cm,marginparwidth=1.75cm]{geometry}
\usepackage[colorlinks = true,
            linkcolor = black,
            urlcolor = black,
            bookmarksopen = true]{hyperref}

\usepackage{amsfonts}
\usepackage{graphicx}

\usepackage{url}
\newtheorem{theorem}{Theorem}

\newtheorem{lemma}{Lemma}

\newtheorem{cor}{Corollary}

\newtheorem{remark}[theorem]{Remark}
\newtheorem{example}[theorem]{Example}

\newtheorem{innercustomgeneric}{\customgenericname}
\providecommand{\customgenericname}{}
\newcommand{\newcustomtheorem}[2]{%
  \newenvironment{#1}[1]
  {%
   \renewcommand\customgenericname{#2}%
   \renewcommand\theinnercustomgeneric{##1}%
   \innercustomgeneric
  }
  {\endinnercustomgeneric}
}

\newcustomtheorem{customthm}{Theorem}

\title[Algebraic integers and palindromes]{Algebraic integers with continued fraction expansions containing palindromes and square roots with prescribed periods}

\author[S. Barbero]{Stefano Barbero} 
\address{DISMA ``Luigi Lagrange'', Politecnico di Torino, Corso Duca degli Abruzzi 24, 10129
Torino, Italy}
\email{stefano.barbero@polito.it}

\author[U. Cerruti]{Umberto Cerruti} 
\address{Dipartimento di Matematica ``G. Peano'', Università di Torino, Via Carlo Alberto 10, 10123, Torino, Italy}
\email{umberto.cerruti@unito.it}

\author[N. Murru]{Nadir Murru} 
\address{Dipartimento di Matematica, Università di Trento, Via Sommarive 14, 38123, Povo (TN), Italy}
\email{nadir.murru@unitn.it}

\author[G. Salvatori]{Giulia Salvatori} 
\address{DISMA ``Luigi Lagrange'', Politecnico di Torino, Corso Duca degli Abruzzi 24, 10129
Torino, Italy}
\email{giulia.salvatori@polito.it}

\date{}

\begin{document}

\begin{abstract}
We present a characterization of the algebraic integers with continued fraction expansions of the form $[a_0, \overline{a_1, \ldots, a_n, s}]$, where $(a_1, \ldots, a_n)$ is a palindrome and $s \in \mathbb{N}_{\geq 1}$.
In particular, we focus on the special case where $(a_1, \ldots, a_n) = (m, \ldots, m)$, providing a detailed characterizations of the corresponding algebraic integers and $s$ in terms of Fibonacci polynomials. 
Then, we derive new expansions of square roots of integers with these periods, given $m$ and $n$. Moreover, we explicitly determine the fundamental solutions of both positive and negative Pell's equations corresponding to this family of integers.
\end{abstract}

\maketitle

\section{Introduction}

Given $D \in \mathbb N$ not a square, it is well known that the continued fraction expansion of $\sqrt{D}$ is
\begin{equation}\label{eq:sD}
\sqrt{D} = [a_0, \overline{a_1, \ldots, a_n, 2a_0}] 
\end{equation}
where $(a_1, \ldots, a_n)$ is a palindrome. Given any positive integer $n$, there exist infinitely many integers $D$ such that $\sqrt{D}$ has a periodic expansion with period of length $n+1$.

Friesen \cite{FR88} proved a more general result, providing a sufficient and necessary condition on the palindromic sequence $(a_1, \ldots, a_n)$ in order that there exists $\sqrt{D}$ having the expansion \eqref{eq:sD} and in this case there are infinitely many of such square roots of integers. Mc Laughlin \cite{LA03} provided an analogue characterization and used these results for addressing also the problem of solving the corresponding Pell's equation with multivariate polynomials. 

Many other authors studied similar problems involving the length of the period of $\sqrt{D}$, the shape of the palindromic sequence and the construction of infinitely many $\sqrt{D}$ with a prescribed period. For instance, Pletser \cite{Ple14} focused on the special case where $(a_1, \ldots, a_n) = (m, \ldots, m)$, for any positive integers $m$, $n$ and studied $D \in \mathbb N$ such that $\sqrt{D} = [a_0, \overline{m, \ldots, m, 2a_0}]$, obtaining $D$ in terms of Fibonacci polynomials. Then he exploited this result for writing explicit expressions of the fundamental solutions of the associated Pell's equation. Das et al. \cite{Das20} studied the expansion of $\sqrt{pq}$, with $p,q$ primes, giving some information on the length of the period and on the parity of the central term in the palindromic sequence. Gawron and Kobos \cite{GK23} proved that there exist infinitely many $k \in \mathbb N$ such that there exist infinitely many $n$ for which the length of the period of $n \sqrt{D}$ is $k$. Rada and Starosta \cite{RS20} found upper and lower bounds for the length of the period of the Mobius transform of $\sqrt{D}$, while Kala and Miska \cite{KM23} found an upper bound when $D$ factorizes in some family of polynomials with integral coefficients. Moreover, they proved that for each positive integer $a$ there exist only finitely many prime numbers $p$ such that $a$ appears an odd number of times in the period of continued fraction of $\sqrt{p}$ or $\sqrt{2p}$, with $p$ prime. Further results can be found in \cite{Koch, Kaw1, Mollin03b, Tom}.

The paper is structured as follows.
In Section \ref{sec:intp}, we give a characterization of all quadratic algebraic integers having continued fraction expansion $[a_0, \overline{a_1, \ldots, a_n, s}]$ where $(a_1, \ldots, a_n)$ is a palindrome and $s \in \mathbb N_{\geq1}$, in Theorem \ref{palindrome} and Corollary \ref{coro}. We also provide an explicit method for finding $s$ and determining the corresponding algebraic integer. This result is not new (since it can be derived from \cite[Theorem~3.1]{Kaw2}, \cite{FR88} and \cite{LA03}), however, we include it alongside our proof (slightly different from the approach that exploits \cite[Theorem~3.1]{Kaw2}) as it constitutes a fundamental component in the subsequent findings. 

In Section \ref{sec:m}, in Theorem \ref{mmm}, we specify this study to the special case $(a_1, \ldots, a_n) = (m, \ldots, m)$, describing the minimal polynomial of such algebraic integers in terms of Fibonacci polynomials. 
Moreover, we find the explicit continued fraction expansion of a infinite family of square roots of integers (Theorems \ref{thm:mne} and \ref{thm:meno}, Remark \ref{rem:rn}), derived from algebraic integers of Theorem \ref{mmm}. 

In Section \ref{sec:pell} we provide the explicit forms for the fundamental solutions of both positive and negative Pell's equations corresponding to these integers.

\section{Algebraic integers with palindromic sequences in their continued fraction expansion} \label{sec:intp}

In the following theorem, we prove that for any palindromic sequence there exist infinitely many algebraic integers containing such sequence in the period of the continued fraction expansion. 

\begin{remark}\emph{
We point out that Theorem \ref{palindrome} is readily deducible from \cite[Theorem~3.1]{Kaw2}. However, we provide our own proof as it explicitly provides the continued fraction expansion and the minimal polynomial of the algebraic integers in the interval $(0,1)$ so that they are effectively given. These results are particularly pertinent to Theorem \ref{mmm}, where we examine the specific case of a palindrome sequence composed of constant elements.}
\end{remark}

\begin{theorem} \label{palindrome}	
Given a finite palindromic sequence $(a_1,a_2,\ldots,a_n)$, with $n \geq 1$, there exist infinitely many algebraic integers $\alpha$ of degree $2$ such that $0<\alpha<1$ and
$$\alpha=[0,\overline{a_1, a_2, \ldots, a_n, w + z k}] $$
for all $k \geq k_1$, where $k_1$ is a constant depending on the sequence $(a_1,a_2,\ldots,a_n)$ and $w, z$ are constants depending on the sequence $(a_1,a_2,\ldots,a_n)$ and $k$.
\end{theorem}

\begin{proof}
Let us consider the finite sequences
\[ (A_0, A_1, \ldots, A_n), \quad (B_0, B_1, \ldots, B_n) \]
defined by
\[ A_0 = 1, \quad A_1 = a_1, \quad B_0 = 0, \quad B_1 = 1 \]
and
\[ A_h = a_h A_{h-1} + A_{h-2}, \quad B_h = a_h B_{h-1} + B_{h-2} \]
for all $2 \leq h \leq n$.
Remembering that
\[ \begin{pmatrix} a_1 & 1 \cr 1 & 0 \end{pmatrix} \cdots \begin{pmatrix} a_h & 1 \cr 1 & 0 \end{pmatrix} = \begin{pmatrix} A_h & A_{h-1} \cr B_h & B_{h-1} \end{pmatrix}  \]
since the sequence $(a_1,a_2,\ldots,a_n)$ is a palindrome, we have $A_{n-1} = B_n$.


Considering now $\alpha = [0, \overline{a_1, \ldots, a_n, s}]$, we have $\alpha = [0, a_1, \ldots, a_n, s + \alpha]$ from which 
\[ \alpha = \cfrac{(s+\alpha) B_n + B_{n-1}}{(s+\alpha) A_n + A_{n-1}} \]
and then
\[ A_n \alpha^2 + s A_n \alpha - s B_n - B_{n-1} = 0. \]
Thus, $\alpha$ is an algebraic integer if and only if $-s B_n - B_{n-1} \equiv 0 \pmod{A_n}$. Since $\gcd(A_n, A_{n-1}) = 1$ and $A_n B_{n-1} - B_n A_{n-1} = (-1)^n$, the previous congruence is equivalent to
\[ (-1)^n s \equiv A_{n-1} B_{n-1} \pmod{A_n}. \]
Thus, considering 
\[ s = (-1)^n A_{n-1} B_{n-1} + k A_n, \]
with $k \in \mathbb Z$, we also have
\[ -s B_n - B_{n-1} = (-1)^{n-1} A_{n-1} B_{n-1}B_n - k A_nB_n - B_{n-1} = A_n ((-1)^{n-1} B_{n-1}^2 - k B_n), \]
from which the minimal polynomial of $\alpha$ is 
$$x^2 + sx + t,$$ 
where
\[ t = (-1)^{n-1} B_{n-1}^2 - k B_n.\]
Clearly, in order to have $\alpha = [0, \overline{a_1,\ldots,a_n,s}]$ we must require $s > 0$ and this is verified if $k \geq k_1$, where
\[k_1 = \left \lceil \cfrac{(-1)^{n-1}A_{n-1}B_{n-1}}{A_n} \right \rceil\]
\end{proof}

The following corollary completes Theorem \ref{palindrome} for the case $\alpha \notin (0,1)$.
\begin{cor}\label{coro}
Given a finite palindromic sequence $(a_1,a_2,\ldots,a_n)$ and $a_0 \in \mathbb{Z}$, with $n \geq 1$, there exist infinitely many algebraic integers $\alpha$ of degree $2$ such that 
$$\alpha=[a_0,\overline{a_1, a_2, \ldots, a_n, w + z k}], $$
where $k$, $w$ and $z$, which do not depend on $a_0$, are the constants given by Theorem \ref{palindrome} associated with $\{\alpha \}= [0,\overline{a_1, a_2, \ldots, a_n, w + z k}]$, the fractional part of $\alpha$.
\end{cor}
\begin{proof}
The ring of integers of a quadratic number field is an ring containing $\mathbb{Z}$. This implies that $\alpha$ is an algebraic integer if and only if its fractional part $\{\alpha\}=\alpha - a_0 = [0,\overline{a_1, a_2, \ldots, a_n, s}]$ is an algebraic integer. The proof follows from Theorem \ref{palindrome}. We point out the fact that, if $x^2 + sx + t$ is the minimal polynomial of $\{\alpha\}$, with $s$ and $t$ given in the proof of Theorem \ref{palindrome}, then $(x-a_0)^2 + s(x-a_0) + t = x^2+(s-2a_0)x+a_0^2-sa_0+t$ is the minimal polynomial of $\alpha$. 
\end{proof}

Clearly, given $a_0 \in \mathbb{Z}$, a palindromic sequence $(a_1, \ldots, a_n)$ and an integer $s$, there exists always a quadratic irrational $\alpha$ such that $\alpha = [a_0, \overline{a_1, \ldots, a_n,s}]$, but in general $\alpha$ is not an algebraic integer. Theorem \ref{palindrome} and Corollary \ref{coro} provide a method for finding all and only the positive integers $s$ such that $\alpha = [a_0, \overline{a_1, \ldots, a_n,s}]$ is an algebraic integer. 

\begin{remark}\emph{
Thanks to Theorem \ref{palindrome} and Corollary \ref{coro} it is possible to give sufficient and necessary conditions on the palindromic sequence $(a_1, \ldots, a_n)$ in order that there exists $\sqrt{D}$ having the expansion \eqref{eq:sD}, which are equivalent to those found by Friesen in \cite{FR88}. 
There exists $D$ such that $\sqrt{D}=[a_0,\overline{a_1, a_2, \ldots, a_n,2a_0}]$ if and only if
\begin{equation}\label{condnecsuff}2a_0 \equiv (-1)^n A_{n-1} B_{n-1} \pmod{A_n}.\end{equation}
If $A_n \equiv 1 \pmod{2}$ then it is always possible to find values of $a_0$ satisfying \eqref{condnecsuff}. If $A_n \equiv 0 \pmod{2}$, then $A_{n-1} \equiv 1 \pmod{2}$ and so the condition of the existence of $a_0$s is equivalent to $B_{n-1} \equiv 0 \pmod{2}$. Therefore, if $B_{n-1} \equiv 0 \pmod{2}$ it is always possible to find such $a_0$s. If $B_{n-1} \equiv 1 \pmod{2}$ there exist such $a_0$s if and only if $A_n \equiv 1 \pmod{2}$. In this case, $A_n \equiv 1+B_n^2 \pmod{2}$, and so $A_n \equiv 1 \pmod{2}$ if and only if $B_n \equiv 0 \pmod{2}$. In conclusion, given a palindromic sequence $(a_1, \ldots, a_n)$, there exists $D$ such that $\sqrt{D}=[a_0,\overline{a_1, a_2, \ldots, a_n,2a_0}]$ if and only if one of the following two conditions holds
\begin{enumerate}
    \item $B_{n-1} \equiv 0 \pmod{2}$;
    \item $B_{n-1} \equiv 1 \pmod{2}$ and $B_{n} \equiv 0 \pmod{2}$.
\end{enumerate}
}
\end{remark}

\begin{remark}\emph{
Depending on the chosen value of $k$ and consequently $s$, the period of the continued fraction expansion of $\alpha = [a_0, \overline{a_1, \ldots, a_n,s}]$ can be smaller than $n+1$, as shown in Example \ref{exampleperiod}. Choosing $s = (-1)^n A_{n-1} B_{n-1} + k A_n$, using the same notation of Theorem \ref{palindrome}, with $k$ such that $s > \max \{a_1, \ldots, a_n \}$, we have that the continued fraction expansion of $\alpha$ has period $n+1$.}
\end{remark}

\begin{example}\emph{
Consider the palindromic sequence $(2,5,5,2)$. The sequences $(A_n)_{n=0}^4$ and $(B_n)_{n=0}^4$ defined in the proof of Theorem \ref{palindrome} are
\[ (1, 2, 11, 57, 125), \quad (0, 1, 5, 26, 57), \]
respectively. Thus, considering 
\[ k \geq \left\lceil \cfrac{(-1)^3 A_{3}B_3}{A_4} \right \rceil = -11\]
we get infinitely many algebraic integers with expansion $[0, \overline{2,5,5,2, s}]$. For instance, considering $k = -11$, we obtain
\[ s = (-1)^4A_3B_3-11A_4 = 107, \quad t = (-1)^3 B_3^2+11B_4 = -49, \]
i.e., $\alpha = [0, \overline{2,5,5,2, 107}]$ where $\alpha = \frac{-107+\sqrt{11645}}{2}$ is the algebraic integer with minimal polynomial $x^2 + 107x - 49$ and $0 < \alpha < 1$. Choosing $s=74$, then $(-1)^4s \not\equiv A_{3}B_{3} \pmod{A_4}$ and $\alpha = [0, \overline{2,5,5,2, 74}] = \frac{-925 + \sqrt{876845}}{25}$ is not an algebraic integer.
}
\end{example}

\begin{example}\emph{Let $a$ and $b$ two positive distinct integers, we define the sequences $S_n$, for $n\ge0$ as follows:
\[
\begin{cases}S_0 = (a) \\
S_1 = (a,b) \\
S_{n} = S_{n-1} \| S_{n-2}, \text{ for }n \ge 2
\end{cases}
  \]
where $\|$ denotes the string concatenation. These sequences are used to define the Fibonacci word (see \cite{Shall}). It is easy to prove that, for $n\ge3$ the following are palindromic sequences
\[\begin{cases}\hat{S}_n = (a,b) \| S_n, \quad \text{if }n \equiv 0 \pmod{2} \\ 
\hat{S}_n = (b,a) \| S_n, \quad \text{if }n \equiv 1 \pmod{2}.
\end{cases}\]
We have that $|\hat{S}_n| = F_{n+2} +2$, where $F_n$ is the $n$-th element of the Fibonacci sequence.
For $n=3$, $\hat{S}_3 = (b,a,a,b,a,a,b)$ is a palindromic sequence of $F_{5} +2 =7$ elements. The sequences $(A_n)_{n=0}^7$ and $(B_n)_{n=0}^7$ defined in the proof of Theorem \ref{palindrome} are
\[ \begin{aligned}(&1, b, ab+1, a^2b+a+b, a^2b^2+2ab+b^2+1, a^3b^2+3a^2b+ab^2+2a+b, \\ &a^4b^2+3a^3b+2a^2b^2+3ab+2a^2+b^2+1,\\ &a^4b^3+4a^3b^2+2a^2b^3+4ab^2+5a^2b+b^3+2a+2b),\end{aligned}\]
and
\[\begin{aligned}(&0,1,a,a^2+1, a^2b+a+b, a^3b+2a^2+ab+1, a^4b+2a^3+2a^2b+2a+b,\\ &a^4b^2+3a^3b+2a^2b^2+3ab+2a^2+b^2+1),\end{aligned}\]
respectively. Thus, considering 
\[ k \geq k_1= \left\lceil \cfrac{(-1)^6 A_{6}B_6}{A_7} \right \rceil\]
we get infinitely many algebraic integers with expansion $[0, \overline{b,a,a,b,a,a,b,s}]$, and minimal polynomial $x^2+sx+t$, where
\[s=(-1)^7A_6B_6 + kA_7 \quad \text{and} \quad t=(-1)^6B_6^2-kB_7.\]
For example, setting $a=1$ and $b=2$, we obtain $k\ge5$. In the case $k=5$ we obtain $s=28$, $t=-724$ and $\alpha=[0,\overline{2,1,1,2,1,1,2,28}]=-14+\sqrt{207}$, with minimal polynomial equal to $x^2+28x-724$.
}
\end{example}

\begin{example}\emph{
Consider the palindromic sequences $(1,1,2m,1,1)$, for $m \in \mathbb N$ not zero. The sequences $(A_n)_{n=0}^5$ and $(B_n)_{n=0}^5$ defined in the proof of Theorem \ref{palindrome} are
\[ (1,1,2,4m+1,4m+3,8m+4), \quad (0,1,1,2m+1,2m+2,4m+3) \]
and 
\[k_1 = \left\lceil \frac{(m+1)(4m+3)}{4m+2} \right\rceil  = m +2.\]
Indeed, 
\[ (m+1)(4m+3) = (m+1)(4m+2) + m + 1 \]
 and since $0 < m+1 < 4m+2$ we have 
 \[ \left\lfloor \frac{(m+1)(4m+3)}{4m+2} \right\rfloor = m + 1. \]
 Thus, taking $k = k_1$, we have \(s = 6m+2\) and we obtain 
\[ -3m-1 + \sqrt{9m^2+9m+3} = [0,\overline{1,1,2m,1,1,6m+2}] \]
whose minimal polynomial is $x^2+(6m+2) x - 3m - 2$. For instance, we have the following expansion of algebraic integers:
\[ -4 + \sqrt{21} = [0,\overline{1,1,2,1,1,8}], \quad -7 + \sqrt{57} = [0,\overline{1,1,4,1,1,14}],\]
\[ -13 + \sqrt{183} = [0,\overline{1,1,6,1,1,20}], \quad -16 + \sqrt{273} = [0,\overline{1,1,8,1,1,26}],\]
and so on.
}
\end{example}

\begin{example}\label{exampleperiod}\emph{
Consider the palindromic sequences $(1,1,2m+1,1,1)$, for $m \in \mathbb N$. The sequences $(A_n)_{n=0}^5$ and $(B_n)_{n=0}^5$ defined in the proof of Theorem \ref{palindrome} are
\[ (1,1,2,4m+3,4m+5,8m+8), \quad (0,1,1,2m+2,2m+3,4m+5) \]
and 
\[k_1 = \left\lceil \frac{(2m+3)(4m+5)}{8m+8} \right\rceil  = m+2.\]
Indeed, 
\[ (2m+3)(4m+5) = (m+1)(8m+8) + 6m+7 \]
 and since $0 < 6m+7 < 8m+8$ we have 
 \[ \left\lfloor \frac{(2m+3)(4m+5)}{8m+8} \right\rfloor = m + 1. \]
Taking $k = k_1$, we have \(s = 2m+1\) and we obtain 
\[ \alpha = [0,\overline{1,1,2m+1,1,1,2m+1}] = [0,\overline{1,1,2m+1}].\]
}
\end{example}

\section{Algebraic integers with periodic continued fractions having all terms but last equal and expansions of associated square roots} \label{sec:m}

We define the sequence of Fibonacci polynomials $(f_h(m))_{h \geq 0}$ (as defined, e.g., in \cite{Ple14}) by
\[ \begin{cases} f_0(m) = 0, \quad f_1(m) = 1 \cr f_h(m) = m f_{h-1}(m) + f_{h-2}(m), \quad \forall h \geq 2 \end{cases}, \]
where $m$ is a fixed integer. They can be defined for negative indices by
\[f_{-h}(m) = (-1)^{h-1}f_{h}(m), \quad \forall h \geq 1.\]

\begin{theorem} \label{mmm}
Consider the periodic continued fraction $\alpha = [0, \overline{m, \ldots, m, s}]$, whose period has length $n+1$, then there exist infinitely many values of $s$ (depending on $m, n$) such that $\alpha$ is an algebraic integer and its minimal polynomial is $x^2 + sx + t$, where
\[ s= (-1)^n f_n(m)f_{n-1}(m) + k f_{n+1}(m), \quad t = (-1)^{n-1} f^2_{n-1}(m) - k f_n(m) \]
for all $k \geq (-1)^{n-1} f_{n-2}(m) + 1$.
\end{theorem}
\begin{proof}
Considering the finite continued fraction $[m, \ldots, m]$ and numerators and denominators of its convergents (as defined also in the proof of Theorem \ref{palindrome}) we have $A_h = f_{h+1}(m)$ and $B_h = f_h(m)$, for all $0 \leq h \leq n$. Thus, by Theorem \ref{palindrome}, it follows that $\alpha$ is an algebraic integer whose minimal polynomial is $x^2 + sx + t$, where
\[ s= (-1)^n f_n(m)f_{n-1}(m) + k f_{n+1}(m), \quad t = (-1)^{n-1} f^2_{n-1}(m) - k f_n(m) \]
for $k \in \mathbb Z$. In this case, for obtaining $s > 0$, we exploit that
\[ f_{j-1}(m) f_i(m) + f_j(m) f_{i+1}(m) = f_{i+j}(m) \quad \forall i,j \in \mathbb Z, \]
from which we obtain
\[(-1)^nf_{n-1}(m)f_n(m) +(-1)^{n-1}f_{n-2}(m)f_{n+1}(m) = f_2(m) = m.\] 
Considering
\[ k_0 = \cfrac{(-1)^{n-1}f_n(m)f_{n-1}(m)}{f_{n+1}(m)} = (-1)^{n-1} f_{n-2}(m) - \cfrac{m}{f_{n+1(m)}}, \]
if we take $k = \lceil k_0 \rceil$, we get $s = m$. Thus, we must have $k \geq (-1)^{n-1} f_{n-2}(m) + 1$ in order to have the period length equal to $n+1$.
\end{proof}

In the following, we focus on the continued fraction expansion (and the corresponding Pell's equations) of $\beta(n, m, k) := \sqrt{s(n, m, k)^2 - 4t(n, m, k)}$, where $s$ and $t$ are given as in Theorem \ref{mmm} (where, when necessary, we highlight the dependence on $n, m, k$).

\begin{lemma}[\cite{Ran73}]\label{lemma:RL}
Given the matrices 
\[ R = \begin{pmatrix} 1 & 1 \cr 0 & 1 \end{pmatrix}, \quad L = \begin{pmatrix} 1 & 0 \cr 1 & 1 \end{pmatrix}, \quad J = \begin{pmatrix} 0 & 1 \cr 1 & 0 \end{pmatrix} \]
we have
\[ \begin{pmatrix} a & 1 \cr 1 & 0 \end{pmatrix} = R^aJ = JL^a \]
and
\[ \prod_{i=0}^\infty \begin{pmatrix} a_i & 1 \cr 1 & 0 \end{pmatrix} = \prod_{i=0}^\infty R^{a_{2i}} L^{a_{2i+1}}. \]
\end{lemma}

\begin{lemma}\label{MRL}
Given the matrices
\[ M = \begin{pmatrix} 2 & 0 \cr 0 & 1 \end{pmatrix}, \quad N = \begin{pmatrix} 1 & 0 \cr 0 & 2 \end{pmatrix}, \quad P = \begin{pmatrix} 2 & 0 \cr 1 & 1 \end{pmatrix}, \quad Q = \begin{pmatrix} 1 & 1 \cr 0 & 2 \end{pmatrix} \]
and any positive integer $e$, we have
\begin{itemize}
\item $MR^e = R^{2e}M$, $NL^e = L^{2e}N$;
\item $PR^e = RLR^{\frac{e-2}{2}}Q$, $QL^e = LRL^{\frac{e-2}{2}}P$, $NR^e = R^{\frac{e}{2}}N$, $ML^e = L^{\frac{e}{2}}M$, if $e$ is even;
\item $PR^e = RLR^\frac{e-1}{2}N$, $QL^e = LRL^{\frac{e-1}{2}}M$, $NR^e = R^{\frac{e-1}{2}}Q$, $ML^e = L^{\frac{e-1}{2}}P$, if $e$ is odd.
\end{itemize}
\end{lemma}
\begin{proof}
The proof is straightforward.
\end{proof}

\begin{lemma}\label{lemma01}
Consider $\alpha(n,m,k) = \frac{1}{2}(\sqrt{s^2-4t} - s)$, with $k \geq (-1)^{n-1}f_{n-2}(m)+1$, we have $1 < 2\alpha(n,1,k) < 2$ and $0 < 2 \alpha(n,m,k) < 1$ for all $m \geq 2$.
\end{lemma}
\begin{proof}
We have that $\alpha(n,m,k) < \frac{1}{2}$ if and only if $2s + 4t + 1 > 0$. We can observe that
\[ 2s + 4t + 1 = 2k (f_{n+1} - 2f_n) + 2(-1)^n f_n f_{n-1} + 4 (-1)^{n-1} f_{n-1}^2 + 1 > - \cfrac{4 f_{n-1}}{f_{n+1}} + 1,  \]
where we omit the dependence on $m$ for the seek of simplicity.
Moreover, using the Cassini identity, we obtain
\[ f_{n+1}(m) - 4 f_{n-1}(m) = (m^2-3) f_{n-1}(m) + m f_{n-2}(m) \]
which is greater than 0 if and only if $m \geq 2$. 

In the case $m = 1$, we can prove that 
\[2s(n,1,k) + 4t(n,1,k) + 1 < 0.\]
Indeed, since $-k \leq (-1)^nf_{n-2}-1$, where here $f_n$'s are the Fibonacci numbers, we have
\begin{align*} 
2s(n,1,k) + 4t(n,1,k) + 1 = - 2kf_{n-2}-(-1)^n2f_{n-3}f_{n-1}+1 \cr
< 2(-1)^nf_{n-2}^2-2f_{n-2}-(-1)^n2f_{n-3}f_{n-1}+1 = -2f_{n-2}-1 < 0.
\end{align*}
\end{proof}

By Lemma \ref{lemma01}, for $m \geq 2$, we know that $0 < 2 \alpha(n,m,k) < 1$ and $\lfloor \beta(n,m,k) \rfloor = s$, since $\beta(n,m,k) = s + 2 \alpha(n,m,k)$. Thus, it is sufficient to study the continued fraction expansion of $2 \alpha(n,m,k)$ in order to know the structure of the period for the expansion of $\beta(n,m,k)$. Similarly, when $m = 1$, it will be sufficient to focus on $2 \alpha(n,1,k) - 1$, since we have $\lfloor \beta(n,1,k) \rfloor = s + 1$.

\begin{theorem} \label{thm:mne}
Given $m,n > 0$ even integers, let $\mathbf v(m) = (\frac{m}{2}, 2m, \frac{m}{2}, 2m, \ldots, \frac{m}{2}, 2m)$ be a finite sequence of length $n$ and let $\mathbf v(m)^R$ be the sequence $\mathbf v(m)$ reversed. Consider also $\mathbf u(m) = (1, 1, \frac{m-2}{2},1, 1, \frac{m-2}{2}, \ldots, 1, 1, \frac{m-2}{2})$ of length $3n$. Then,
\begin{enumerate}
\item $\beta(n,m,k) = [s, \overline{\mathbf v(m), \frac{s}{2}, \mathbf v(m)^R, 2s}]$, when $k$ is even;
\item $\beta(n,m,k) = [s, \overline{\mathbf v(m), \frac{s-1}{2}, \mathbf u(m), 1, 1, \frac{s-1}{2}, \mathbf v(m)^R, 2s}]$, when $k$ is odd and $m > 2$;
\item $\beta(n,2,k) = [s, \overline{\mathbf v(2), \frac{s-1}{2}, 1, \mathbf{2}, 1, \frac{s-1}{2}, \mathbf v(2)^R, 2s}]$, when $k$ is odd, where $\mathbf 2$ is a finite sequence of all $2$ of length $n$.
\end{enumerate}
\end{theorem}
\begin{proof}
When $k$ is even, since also $m, n$ are even, we have $s$ even. By Lemma \ref{lemma:RL}, the continued fraction expansion of $\alpha(n,m,k)$ can be represented also by the infinite product of the matrices $A$, where
\[ A = \underbrace{L^mR^m \cdots L^mR^m}_{n}L^s\underbrace{R^mL^m \cdots R^mL^m}_{n}R^s \]
By Lemma \ref{MRL}, we know that
\[MR^e = R^{2e}M, \quad ML^e = L^{e/2}M, \]
when the exponent $e$ is even. Thus,
\[ MA = \underbrace{L^{m/2}R^{2m} \cdots L^{m/2}R^{2m}}_{n}L^{s/2}\underbrace{R^{2m}L^{m/2}\cdots R^{2m}L^{m/2}}_{n}R^{2s} M = B, \]
i.e., the continued fraction expansion of $2\alpha(n,m,k)$ can be represented by the infinite product of matrices $B$ and consequently
\[\beta(n,m,k) = [s, \overline{\mathbf v(m), \frac{s}{2}, \mathbf v(m)^R, 2s}].\]
When $k$ is odd (consequently also $s$ is odd) and $m \geq 4$, for evaluating $MA$, we also need the following identities from Lemma \ref{MRL}:
\[ PR^e = RLR^{\frac{e-2}{2}}Q, \quad QL^e = LRL^{\frac{e-2}{2}}P \]
for $e$ even and
\[ ML^e = L^{\frac{e-1}{2}}P \]
for $e$ odd. Now, 
\small{
\[\begin{aligned}MA = &\underbrace{L^{m/2}R^{2m}\cdots L^{m/2}R^{2m}}_{n}L^{(s-1)/2}\underbrace{RLR^{(m-2)/2}LRL^{(m-2)/2}\cdots RLR^{(m-2)/2}LRL^{(m-2)/2}}_{3n} \cdot \\
&\cdot RLR^{(s-1)/2}N, \end{aligned}\]}
then exploiting also
\[ NL^e=L^{2e}N, \quad NR^e = R^{e/2}N \]
for $e$ even, we can evaluate $NA$ and finally obtaining
\[\beta(n,m,k) = [s, \overline{\mathbf v(m), \frac{s-1}{2}, \mathbf u(m), 1, 1, \frac{s-1}{2}, \mathbf v(m)^R, 2s}]\]
where the period has length $5n+5$.
For the final case $k$ odd and $m = 2$, exploiting
\[ PR^2L^2 = RL^2RP, \quad NL^2R^2 = L^4RN, \]
it is possible to conclude the proof.
\end{proof}

\begin{theorem} \label{thm:meno}
Given $m, n > 0$ integers, with $m$ even and $n$ odd, then
\[ \beta(n,m,k) = [s, \overline{\mathbf w(m), 2s}], \]
with $k \geq f_{n-2}(m) + 1$, where $\mathbf w(m) = (\frac{m}{2}, 2m, \frac{m}{2}, 2m, \ldots, \frac{m}{2}, 2m, \frac{m}{2})$ has length $n$.
\end{theorem}
\begin{proof}
Since $m$ is even and $n$ odd, $s$ is even (independently from the parity of $k$). By Lemma \ref{lemma:RL}, the continued fraction expansion of $\alpha(n,m,k)$ can be represented also by the product $\prod A$, where
\[ A = \underbrace{L^mR^m \cdots L^mR^mL^m}_{n}R^s, \]
since $n$ is odd. Now,
\[ MA = \underbrace{L^{m/2}R^{2m} \cdots L^{m/2}R^{2m}L^{m/2}}_{n}R^{2s}, \]
i.e.,
\[2\alpha(n,m,k) = [0, \overline{\frac{m}{2}, 2m, \ldots ,\frac{m}{2}, 2m, \frac{m}{2},2s}]\]
where the period has length $n+1$ and the continued fraction expansion of $\beta(n,m,k)$ follows.
\end{proof}

\begin{remark}\label{rem:rn}\emph{
The remaining cases for the expansion of $\beta(n,m,k)$ can be proved similarly to Theorems \ref{thm:mne} and \ref{thm:meno} and the resulting expansions are slightly different. Observe that it is convenient to consider the different cases $m=1$ and $m\geq3$ odd combined with the possible values of $n$ modulo 6. This is due to the fact that the parity of $s$ and $k$ changes depending on these values for $m$ and $n$.}
\end{remark}

\begin{example}\emph{
In \cite{MC02}, the authors introduce two families of quadratic irrationals. The first are the creepers which are sets of quadratic irrationals $\{ \alpha_n \}_{n \in \mathbb N}$ such that the length of the period of the continued fraction expansion of $\alpha_n$ is $an+b$, where $a,b \in \mathbb{N}$. Using what proved far above, it is possible to construct some particular families of creepers. The following is a family of creepers for $a=4$, $b=2$ and $m$ a positive even integer
\[\{ \alpha_n = \beta(2n,m,2) \}_{n \in \mathbb N}.\]}

\emph{The second family introduced in \cite{MC02} are the sleepers: sets of quadratic irrationals $\{ \gamma_n \}_{n \in \mathbb N}$
having a continued fraction expansion whose length of the period is constant. The authors also constructed some particular families of sleepers. Further ones can be found in \cite{Mollin03, MG04}.\\ 
Thanks to the previous results, we are able to provide new families of sleepers. For instance, fixed the values of $n$ and $k$, we have that  
$\{\beta(n,2x,k): x \in \mathbb N_{\geq1}\}$ are families of sleepers, with $k$ satisfying the conditions of Theorem \ref{mmm}.\\
Let us consider $n=2$ and $k=2$, in this case $\beta(2,2x,2) = [s, \overline{x, 4x, \frac{s}{2}, 4x, x,2s}]$, for all $x \in \mathbb N_{\geq1}$ and we obtain the following sleepers:
\[ \beta(2,2,2) = 2 \sqrt{41} = [12, \overline{1,4,6,4,1,24}] \]
\[\quad \beta(2,4,2) = 2\sqrt{370} = [38, \overline{2,8,19,8,2,76}] \]
\[\quad \beta(2,6,2) = 2 \sqrt{1613} = [80, \overline{3,12,40,12,3,160}] \]
\[\beta(2,8,2) = 2 \sqrt{4778} = [138, \overline{4,16,69,16,4,276}] \]
\[ \beta(2,10,2) = 2 \sqrt{11257} = [212, \overline{5,20,106,20,5,424}], \ldots \]
\[ \beta(2,100,2) = 2 \sqrt{101022802} = [20102, \overline{50,200,10051,200,50,40204}]  , \ldots\]
We can observe that also $\beta(2,2x,2y) = [s, \overline{x, 4x, \frac{s}{2}, 4x, x,2s}]$, i.e., we can construct families of sleepers depending on two variables. Varying $n$, we obtain different lengths of the periods. In particular $\beta(n,2x,2y)$, for all $x,y \in \mathbb N_{\geq1}$, has period of length $4n +2$. 
}
\end{example}

\section{Fundamental solutions of some Pell's equations} \label{sec:pell}

In this section, we explicitly exhibit the fundamental (minimal) solution of the Pell's equations $X^2 - DY^2 = \pm 1$ in terms of Fibonacci polynomials, when $D= s^2(n,m,k)-4t(n,m,k)$. 

\begin{lemma} \label{lemma:-4}
Let $D \in \mathbb N$ be not square and $D \not \equiv 0 \pmod 4$, if $(u,v)$ is the minimal solution of the Diophantine equation $X^2 - DY^2 = -4$, then 
$$\left(\frac{u}{2}, \frac{v}{2} \right) \quad \text{or} \quad \left(\frac{1}{2}(u^3 + 3u), \frac{1}{2}(u^2+1)v\right)$$ 
is the minimal solution of the negative Pell's equation $X^2 - DY^2 = -1$.
\end{lemma}
\begin{proof}
Let $\mathcal N(\cdot)$ be the norm of the quadratic field $\mathbb Q(\sqrt{D})$, by hypothesis $\mathcal N(u + v \sqrt{D}) = -4$ and consequently 
$$\mathcal N\left(\frac{1}{2}(u+v\sqrt{D})\right) = -1, \quad \mathcal N\left(\frac{1}{8}(u+v\sqrt{D})^3\right) = -1.$$
Exploiting that $v^2 D = u^2 + 4$, we have
\[ \cfrac{u^3 + 3uv^2D}{8} = \cfrac{u^3 + 3u}{2} \]
and
\[ \cfrac{3 u^2 v +v^3 D}{8} = \cfrac{(u^2+1)v}{2}.\]
Thus, if $u, v$ are both even, then $\left(\frac{1}{2}u, \frac{1}{2}v \right)$ is the minimal solution of the negative Pell's equation. If $u$ is odd, then $\left(\frac{1}{2}(u^3 + 3u), \frac{1}{2}(u^2+1)v\right)$ is the minimal solution of the negative Pell's equation. Moreover, we can observe that we can not have the situation where $u$ is even and $v$ is odd, because this situation implies $D \equiv 0 \pmod 4$.
\end{proof}

\begin{theorem} \label{teo:pellneg}
Consider $D= s^2(n,m,k)-4t(n,m,k)$ with $n$ even and $k$ odd and define $T_n := f_{n+1}s+2f_n$, for all $n \geq 0$. If $(x,y)$ is the minimal solution of the negative Pell's equation $X^2 - D Y^2 = -1$, then
\begin{itemize}
\item[a)] If $m$ is even, then 
\[ x = \cfrac{(T_n^2+3)T_n}{2}, \quad y = \cfrac{(T_n^2+1)f_{n+1}}{2}. \]
\item[b)] If $m$ is odd and $n \equiv 0,4 \pmod 6$, then
\[ x = \cfrac{(T_n^2+3)T_n}{2}, \quad y = \cfrac{(T_n^2+1)f_{n+1}}{2}. \]
\item[c)] If $m$ is odd and $n \equiv 2 \pmod 6$, then
\[ x = \cfrac{T_n}{2}, \quad y = \cfrac{f_{n+1}}{2}. \]
\end{itemize}
\end{theorem}
\begin{proof}
Let us observe that supposing $n$ even the negative Pell's equation has solutions, since the length of the period of the expansion of $\sqrt{D}$ is odd. Moreover, $k$ odd is equivalent to $s$ odd and consequently $D \not \equiv 0 \pmod 4$.
Remembering that, for $n$ even, we have
\[ s(n,m,k) = f_{n+1}(m) k + f_n(m) f_{n-1}(m),\quad t(n,m,k) = -f_{n-1}^2(m) - k f_n(m), \]
we obtain
\begin{align*}
D &= \cfrac{f_{n+1}^2s^2+4f_nf_{n+1}^2k+4f_{n-1}^2f_{n+1}^2}{f_{n+1}^2} = \\
&= \cfrac{f_{n+1}^2s^2+4f_nf_{n+1}(f_{n+1}k+f_nf_{n-1})+ 4 f_{n-1}f_{n+1}(f_{n-1}f_{n+1}-f_n^2)}{f_{n+1}^2} = \\
&=\cfrac{f_{n+1}^2s^2+4f_nf_{n+1}s+4f_{n-1}f_{n+1}}{f_{n+1}^2} =\cfrac{(f_{n+1}s + 2f_n)^2+4}{f_{n+1}^2} = \cfrac{T_n^2 + 4}{f_{n+1}^2}, 
\end{align*}
from which $(T_n, f_{n+1})$ is the minimal solution of $X^2 - D Y^2 = -4$.

When $m$ is even, we have $f_h \equiv f_{h-2} \pmod 2$, thus $f_{n+1}$ and $T_n$ are odd and by Lemma \ref{lemma:-4}, the minimal solution of the negative Pell's equation is
\[ x = \cfrac{1}{2}(T_n^2+3)T_n, \quad y = \cfrac{1}{2}(T_n^2+1)f_{n+1}. \]
When $m$ is odd, we have $f_h \equiv f_{h-1} + f_{h-2} \pmod 2$ and
\[ \begin{cases} \text{$f_h$ even, when $h \equiv 0 \pmod 3$} \cr \text{$f_h$ odd, when $h \equiv 1, 2 \pmod 3$} \end{cases} \]
Thus, when $n \equiv 0 \pmod 6$, we have $f_n$ even, $f_{n+1}$ odd and $T_n$ odd. Similarly, when $n \equiv 4 \pmod 6$, we have $f_n$ odd, $f_{n+1}$ odd and $T_n$ odd. In both cases, by Lemma \ref{lemma:-4}, we conclude that the minimal solution of the negative Pell's equation is again
\[ x = \cfrac{1}{2}(T_n^2+3)T_n, \quad y = \cfrac{1}{2}(T_n^2+1)f_{n+1}. \]
Finally, when $m$ is odd and $n \equiv 2 \pmod 6$, we have $f_n$ odd, $f_{n+1}$ even and $T_n$ even, thus by Lemma \ref{lemma:-4} we obtain 
\[ x = \cfrac{T_n}{2}, \quad y = \cfrac{f_{n+1}}{2} . \]
\end{proof}

For obtaining the solutions of the Pell's equation $X^2 - D Y^2 = 1$, we still need the following lemma.

\begin{lemma} \label{lemma:4}
Let $D \in \mathbb N$ be not square, if $(u,v)$ is the minimal solution of the Diophantine equation $X^2 - DY^2 = 4$, then 
$$\left(\frac{u}{2}, \frac{v}{2} \right), \quad \left( \cfrac{2u^2-4}{4}, \cfrac{uv}{2} \right), \quad \text{or} \quad \left(\frac{1}{2}(u^3 + 3u), \frac{1}{2}(u^2+1)v\right)$$ 
is the minimal solution of the negative Pell's equation $X^2 - DY^2 = 1$.
\end{lemma}
\begin{proof}
By hypothesis $\mathcal N(u + v \sqrt{D}) = 4$ and consequently
\[ \mathcal N \left(\cfrac{1}{2}(u+v\sqrt{D})\right) = \mathcal N \left(\cfrac{1}{4}(u+v\sqrt{D})^2\right) = \mathcal N \left(\cfrac{1}{8}(u+v\sqrt{D})^3\right) = 1. \]
Thus, depending on the parity of $u$ and $v$ the minimal solution of $X^2 - DY^2 =1$ is 
$$\left(\frac{u}{2}, \frac{v}{2} \right), \quad \left( \cfrac{2u^2-4}{4}, \cfrac{uv}{2} \right), \quad \text{or} \quad \left(\frac{1}{2}(u^3 + 3u), \frac{1}{2}(u^2+1)v\right)$$ 
where we also exploited that $Dv^2 = u^2 - 4$.
\end{proof}

\begin{theorem} The minimal solution of $X^2 - DY^2 = 1$, with $D = s^2(n,m,k)-4t(n,m,k)$ and $k \geq f_{n-2}(m) + 1$, is
\begin{itemize}
\item[a)] $(2x^2+1, 2xy)$, where $(x,y)$ as in Theorem \ref{teo:pellneg}, for $n$ even and $k$ odd
\item[b)] $\left( \cfrac{T_n}{2}, \cfrac{f_{n+1}}{2} \right)$, for $m$ even and $n$ odd or $m$ odd and $n \equiv 2 \pmod 3$ odd
\item[c)] $\left( \cfrac{T_n^3-3T_n}{2}, \cfrac{(T_n^2-1)f_{n+1}}{2} \right)$, for $m$ odd and $n \equiv 0, 1 \pmod 3$ odd.
\end{itemize}
\end{theorem}
\begin{proof}
When $n$ is even, we deduce the minimal solution of the Pell's equation by Theorem \ref{teo:pellneg}.
 
Let us consider now $n$ odd. In this case, we obtain 
\[ D = \cfrac{T_n^2-4}{f_{n+1}^2}, \]
from which $(T_n, f_{n+1})$ minimal solution of $X^2- DY^2 = 4$.
Thus, for $m$ even or $m$ odd and $n \equiv 2 \pmod 3$, we can observe that $T_n$ and $f_{n+1}$ are both even and by Lemma \ref{lemma:4} we conclude.

When  $m$ odd and $n \equiv 0, 1 \pmod 3$, we can observe that $T_n$ and $f_{n+1}$ are both odd, but $T_n^3-3T_n$ and $(T_n^2-1)f_{n+1}$ are even, so that we obtain the minimal solution of the Pell's equation applying Lemma \ref{lemma:4}.
\end{proof}

\section*{Acknowledgments}
The first and third authors are members of GNSAGA of INdAM.
The third author acknowledges support from Ripple’s University
Blockchain Research Initiative.
The fourth author is partially supported by project SERICS (PE00000014 - https://serics.eu) under the MUR National Recovery and Resilience Plan funded by European Union - NextGenerationEu.


\begin{thebibliography}{}

\bibitem{Das20} S. Das, D. Chakraborty, A. Saikia, On the period of the continued fraction of $\sqrt{pq}$, Acta Arithmetica \textbf{196}, (2020), 291--302.

\bibitem{FR88} C. Friesen, On continued fractions of given period, Proceedings of the American Mathematical Society \textbf{103}, (1988), 9--14.

\bibitem{GK23} F. Gawron, T. Kobos, On length of the period of the continued fraction of $n \sqrt{d}$, International Journal of Number Theory \textbf{19}, (2023), 2089--2099.

\bibitem{Koch} F. Halter-Koch, Continued fractions of given symmetric period, Fibonacci Quarterly \textbf{29}, (1991), 298--303.

\bibitem{KM23} V. Kala, P. Miska, On continued fraction partial quotients of square roots of primes, Journal of Number Theory \textbf{253}, (2023), 215--234.

\bibitem{Kaw1} F. Kawamoto, K. Tomita, Continued fractions with even period and an infinite family of real quadratic fields of minimal type, Osaka Journal of Mathematics \textbf{46}, (2009), 949--993.

\bibitem{Kaw2} F. Kawamoto, K. Tomita, Continued fractions and certain real quadratic fields of minimal type, Journal of the Mathematical Society of Japan \textbf{60}, (2008), 865--903.

\bibitem{LA03} J. Mc Laughlin, Multi-variable polynomial solutions to Pell's equation and fundamental units in real quadratic fields, Pacific Journal of Mathematics \textbf{210}, (2003), 335--349.

\bibitem{Mollin03} R. A. Mollin, Infinite families of Pellian polynomials and their continued fraction expansions, Results in Mathematics \textbf{43}, (2003), 300--317. 

\bibitem{Mollin03b} R. A. Mollin, Construction of families of long continued fractions revisited, Acta Mathematica Academiae Paedagogicae Nyiregyhàziensis \textbf{19}, (2003), 175--181.

\bibitem{MC02} R. A. Mollin, K. Cheng, Beepers, creepers, and sleepers, International Journal of Mathematics \textbf{2}, (2002), 951--956.

\bibitem{MG04} R. A. Mollin, B. Goddard, A description of continued fraction expansions of quadratic surds represented by polynomials, Journal of Number Theory \textbf{107}, (2004), 228--240.

\bibitem{Ple14} V. Pletser, On continued fraction development of quadratic irrationals having all periodic terms but last equal and associated general solutions of the Pell equation, Journal of Number Theory \textbf{136} (2014), 339--353.

\bibitem{RS20} H. Rada, S. Starosta, Bounds on the period of the continued fraction after a Mobius transformation, Journal of Number Theory \textbf{212}, (2020), 122--172

\bibitem{Ran73} G. N. Raney, On continued fractions and finite automata, Mathematische Annalen \textbf{206} (1973), 265--284.

\bibitem{Tom} K. Tomita, Explicit Representation of Fundamental Units of
Some Real Quadratic Fields, II, Journal of Number Theory \textbf{63}, (1997), 275--285.

\bibitem{Shall} J. P. Allouche, J. Shallit, Automatic sequences: theory, applications, generalizations, Cambridge university press, (2003), 212--213.


\end{thebibliography}
\end{document}